\mag=\magstep1  
\advance\voffset-3cm
\advance\hoffset-2cm
\documentclass{article}

\usepackage{amsmath,amssymb,amsthm}
\usepackage{pstricks}

\newtheorem{theorem}{Theorem}[section]
\newtheorem{lemma}[theorem]{Lemma}

\newtheorem{proposition}[theorem]{Proposition}

\begin{document}
\title{The travel time in a finite box\\ 
in supercritical Bernoulli percolation}


\author{Rapha\"el Cerf\\
Universit\'e Paris Sud and IUF}



\maketitle
\begin{abstract}
We consider the standard site percolation model on the three dimensional
cubic lattice.
Starting solely with the hypothesis
that $\theta(p)>0$, we prove that, for any $\alpha>0$, there
exists $\kappa>0$ such that, with probability larger than 
$1-1/n^\alpha$, every pair of vertices inside the box $\Lambda(n)$
are joined by a path having at most
$\kappa(\ln n)^2$ closed sites.
\end{abstract}



 \def \Z {{\mathbb Z}}
 \def \R {{\mathbb R}}
 \def \C {{\mathbb C}}
 \def \cE {{\mathcal E}}
 \def \cF {{\mathcal F}}
 \def \cC {{\mathcal C}}
 \def \cT {{\mathcal T}}
 \def \cP {{\mathcal P}}
 \def \cS {{\mathcal S}}
 \def \bC {{\overline C}}
 \def \bE {{\overline E}}
 \def \bA {{\overline A}}
 \def \bB {{\overline B}}
 \def \ta {\text{two--arms}}
 \def \N {{\mathbb N}}
 \def \P {{\mathbb P}}
 \def \E {{\mathbb E}}

\newcommand{\bd}{\partial\, } 
\newcommand{\din}{\partial^{\, in}}
\newcommand{\dout}{\partial^{\, out}}
\newcommand{\dini}{\partial^{\, in}_\infty}
\newcommand{\douti}{\partial^{\, out}_\infty}
\newcommand{\dine}{\partial^{\, in}_{ext}}
\newcommand{\doutie}{\partial^{\, out,ext}_{\infty}}
\newcommand{\doute}{\partial^{\, out,ext}}
\newcommand{\dedge}{\partial^{\, edge}}
\newcommand{\dexte}{\partial^{ext,\, edge}}
\newcommand{\dstar}{\partial^{\, *}}
\newcommand{\dcirc}{\partial^{\, o}}

\section{Introduction}
We consider the site percolation model on $\Z^3$.
Each site is declared open with probability~$p$
and closed with probability~$1-p$, and the sites
are independent.
One of the most important problems in percolation is to prove that, in
three dimensions, there is no infinite cluster at the critical point.
The most promising strategy so far seems to perform a renormalization
argument \cite{GR}. The missing ingredient is a suitable construction
helping to define a good block, starting solely with the hypothesis
that $\theta(p)>0$. Our main result here is an estimate on the 
travel time in a finite box under the hypothesis that $\theta(p)>0$.
For $n\in\N$, we denote by $\Lambda(n)$ the cubic box
$\Lambda(n)=[-n,n]^3$.
\begin{theorem}
\label{fm}
Let $p$ be such that $\theta(p)>0$ and
let $\alpha>0$. There exists a constant 
$\kappa$, depending on $\alpha$ and $p$,
such that
$$\forall n\geq 2\qquad 
P
\left(\,
\lower 12pt
\vbox{
\hbox{every pair of vertices of the box $\Lambda(n)$}
\hbox{are joined by a path in $\Lambda(n)$ having}
\hbox{\quad\, at most $\kappa(\ln n)^2$ closed sites}
}
\,\right)
\,\geq\,1-\frac{1}{n^\alpha}\,.$$
\end{theorem}
\noindent
This result can be recast in the language of first passage percolation.
If we declare that the travel time is null through an open site and
one through a closed site, and if we denote by
$T_{\Lambda(n)}(x,y)$ the travel time between two points $x,y$ in $\Lambda(n)$,
that is, the infimum of the travel time over
all the paths joining $x$ and $y$ in $\Lambda(n)$, then the above estimate can
be rewritten as
$$\forall n\geq 2\qquad 
P
\big(\,
\forall x,y\in\Lambda(n)\quad
T_{\Lambda(n)}(x,y)\,\leq\,\kappa(\ln n)^2\,\big)
\,\geq\,1-\frac{1}{n^\alpha}\,.$$
The bound $\kappa (\ln n)^2$ is probably not optimal. 
If we start with the hypothesis
that
$p>p_c$, then we get a bound of order $\kappa\ln n$ with the help
of the slab technology. The goal of the game here is to see what we can get
starting only with the hypothesis that $\theta(p)>0$.
The proof relies essentially on the BK and the FKG inequalities
for the probabilistic part (see \cite{GR}), 
and on a tiling of the sphere into $48$ spherical triangles
for the geometric part.
The vertices of these triangles are the vertices of a
Catalan solid called 
the disdyakis dodecahedron or the 
hexakis octahedron (see 
\cite{HO}, page 54 top left for a picture, or \cite{WK}).
This solid is the dual of one of the Archimedean solids, the
great rhombicuboctahedron.
We could write a proof using only cubes, however it would require
more geometric computations. 
With the help of the tiling of the sphere into spherical triangles,
we can build in a straightforward way a path
converging at geometric speed to a prescribed target.
The main point is that the spherical triangles have a diameter strictly
less than one.

\section{Basic notation}
Two sites $x,y$ of the lattice $\Z^3$ are said to be connected if they
are nearest neighbours, i.e., if $|x-y|=1$.
For $x\in\Z^3$, we denote by $C(x)$ the open cluster
containing~$x$, i.e., the connected component of the
set of the open sites containing~$x$.
If $x$ is closed, then $C(x)=\varnothing$.
Let $A$ be a subset of $\Z^3$. We define 
its internal boundary
$\din A$
and its external boundary 
$\dout A$
by
$$\displaylines{
\din A\,=\, \big\{\,x\in A:\exists y\in A^c\quad |x-y|=1\,\big\}\,,\cr
\dout A\,=\, \big\{\,x\in A^c:\exists y\in A\quad |x-y|=1\,\big\} 
\,.
}$$
For $x$ a point in $\Z^3$, 
the distance $d(x,A)$ between $x$ and $A$ is defined as
$$d(x,A)\,=\,
\inf_{y\in A} \,|x-y|\,.$$
Recall that a path 
$z_0,\dots,z_r$
is a sequence of sites such that each site is a 
neighbour of its predecessor:
$$\forall i\in\{\,0,\dots, r-1\,\}\qquad |z_{i+1}-z_i|=1\,.$$
Let $A$ be a subset of $\Z^3$. 
For $x,y$ in $A$, we define the travel time $T_A(x,y)$ between $x$ and $y$
in $A$ by
$$T_A(x,y)\,=\,\inf\,\Big\{\,\sum_{i=0}^r1_{z_i\text{ closed}}:
z_0,\dots,z_r\text{ path in }\Lambda
\text{ from }z_0=x\text{ to }z_r=y\,\Big\}\,.$$
For $x$ in $A$ and $E$ a subset of $A$, we define the travel time 
between $x$ and $E$
in $A$ by
$$T_A(x,E)\,=\,\inf_{y\in E}\,T_A(x,y)\,.$$
\section{An application of the BK inequality}
A routine application of the BK inequality gives a control on the travel
time until the infinite cluster, and this yields a control on the travel
time to exit a finite domain.
\begin{lemma}
\label{BK}
Let $A$ be a finite subset of $\Z^3$ and let $x\in A$.
We have
$$\forall k\geq 1\qquad P\big(\,
T_A(x,\din A)\,\geq \,k
\,\big)\,\leq\,
(1-\theta(p))^k\,.$$
\end{lemma}
\begin{proof}
Let $A$ be a finite subset of $\Z^3$ and let $x\in A$.
The event $\{\,x\longleftrightarrow\infty\,\}$ is included in the event
$\{\,x\longleftrightarrow\din A\,\}$, thus
$$P\big(T_A(x,\din A)=0\big)\,\geq\,
P\big( x\longleftrightarrow\infty\big)\,=\,\theta(p)\,,$$
or, passing to the complementary event,
$$P\big(T_A(x,\din A)\geq 1\big)\,\leq\,
1-\theta(p)\,.$$
In fact, if
$T_A(x,\din A)\geq 1$ and $C(x)$ is not empty, 
then $\dout C(x)$, the outer boundary of the
open cluster of $x$, contains a set of closed sites which separates $x$ from
$\infty$. We iterate next this argument.
We set $C_0(x)=C(x)$ and we define successively, for $k\geq 0$,
$$C_{k+1}(x)\,=\,
C_{k}(x)
\,\cup\,
\dout C_{k}(x)
\,\cup\,
\big\{\,y\in\Z^3: y\longleftrightarrow
\dout\big(
\dout C_{k}(x)\big)
\,\big\}\,.$$
It follows directly from this construction that
$$\forall k\geq 0\qquad
C_k(x)\,=\,\big\{\,y\in\Z^3:
T_{\Z^3}(x,y)\,\leq \,k\,\big\}\,.$$
Therefore, 
$$T_A(x,\din A)\,\geq \,k\qquad
\Longrightarrow\qquad
C_{k-1}(x)\,\subset\,A\setminus\din A\,,
\quad
\dout C_{k-1}(x)\,\subset\,A
\,.$$
The set $C_k(x)$ becomes infinite when it meets the infinite open cluster. 
Whenever $C_k(x)$ is finite, its outer boundary
$\dout C_k(x)$ contains a set 
of closed sites which separates $x$ from
$\infty$. This set realizes the event 
$\big\{\, x\kern 4pt\not\kern-5pt\longleftrightarrow\infty\,\big\}$.
Moreover, by construction, the sets
$\dout C_k(x)$, $k\geq 0$, are pairwise disjoint.
Thus we have
$$\big\{\,
T_A(x,\din A)\,\geq \,k
\,\big\}\,\subset\,
\big\{\, x\kern 4pt\not\kern-5pt\longleftrightarrow\infty\,
\text{ occurs disjointly $k$ times}\,\big\}\,.$$
Applying the BK inequality (see for instance \cite{GR}), we conclude that
$$P\big(\,
T_A(x,\din A)\,\geq \,k
\,\big)\,\leq\,
P\big(\, x\kern 4pt\not\kern-5pt\longleftrightarrow\infty\,\big)^k\,\leq\,
(1-\theta(p))^k\,$$
as required.
\end{proof}
\section{Cubic boxes}
We consider here the case of a cubic box $\Lambda$ centered at the origin.
Let $F_i$, $1\leq i\leq 6$, be the faces of $\Lambda$. 
Each face $F_i$ is a square, which is itself
the union of four squares $F_i^j$, 
$1\leq j\leq 4$. Each of these squares shares a vertex 
with a vertex of $\Lambda$
and admits the center of $F_i$ as another vertex.
We have
$$T_\Lambda(0,\din\Lambda)\,=\,
\min_{1\leq i\leq 6}\min_{1\leq j\leq 4}
T_\Lambda(0,F_i^j)$$
and, by the FKG inequality,
$$\displaylines{
P\big(
T_\Lambda(0,\din\Lambda)\geq k\big)
\,=\,
P\big(\forall i\in\{\,1,\dots,6\,\}\quad
\forall j\in\{\,1,\dots,4\,\}\quad
T_\Lambda(0,F_i^j)\geq k\big)\cr
\,\geq\,
\prod_{1\leq i\leq 6}\prod_{1\leq j\leq 4}
P\big( T_\Lambda(0,F_i^j)\geq k\big)
\,=\,
P\big( T_\Lambda(0,F_1^1)\geq k\big)^{24}\,.
}$$
The last inequality is a consequence of the symmetry of the model,
indeed the random variables
$T_\Lambda(0,F_i^j)$, 
$1\leq i\leq 6$, $1\leq j\leq 4$, are identically distributed.
It follows then from lemma~\ref{BK} that
$$P\big( T_\Lambda(0,F_1^1)\geq k\big)\,\leq\,
(1-\theta(p))^{k/24}\,.$$
We deal next with translated boxes. If $\Gamma$ is a cubic box, we denote 
its faces by
$F_i(\Gamma)$, $1\leq i\leq 6$, and we denote by
$F_i^j(\Gamma)$, 
$1\leq i\leq 6$, $1\leq j\leq 4$, the tiling of the faces of $\Gamma$
into four squares.
We consider the event $\cE(\Lambda,k)$ defined as follows:
for every cubic box $\Gamma$ included in $\Lambda$, whose center
is a point of $\Z^3$, whose sidelength
is an integer, the center of $\Gamma$ can be joined to each of the $24$ squares
on the faces of $\Gamma$ with a path having at most $k$ closed sites.
More precisely,
$$
\cE(\Lambda,k)\,=\,
\big\{\,\forall \Gamma=x+\Lambda(m)\subset\Lambda\quad
T_\Lambda(x,F_i^j(\Gamma))\leq k
\text{ for } 1\leq i\leq 6,\,1\leq j\leq 4\,\big\}
\,.
$$
\begin{proposition}
\label{cube}
Let $p$ be such that $\theta(p)>0$ and
let $\alpha>0$. There exists a constant 
$c$, depending on $\alpha$ and $p$,
such that
$$\forall n\geq 2\qquad 
P\big( \cE(\Lambda(n),c\ln n)\big)
\,\geq\,1-\frac{c}{n^\alpha}\,.$$
\end{proposition}
\begin{proof}
Let us estimate the probability of the complement of the event
$\cE(\Lambda,k)$:
\begin{multline*}
P\big( \cE(\Lambda,k)^c\big)\,=\,
P\big(\exists\, \Gamma=x+\Lambda(m)\subset\Lambda\quad
\exists\,i,j\quad
T_\Lambda(x,F_i^j(\Gamma))> k\big)
\cr
\,\leq\,
\sum_{x\in\Lambda}
\sum_{m:x+\Lambda(m)\subset\Lambda}
\sum_{i,j}
\,\,
P\big(
T_\Lambda(x,F_i^j(
x+\Lambda(m)
))> k
\big)\,.
\end{multline*}
By translation invariance and symmetry, 
the probability 
inside the sum
depends neither on $x$ nor on $i,j$.
The number of subboxes $\Gamma$
included in $\Lambda$ is bounded by $|\Lambda|\times\text{diameter}(\Lambda)$, 
so we conclude with the help of the previous estimate that
$$P\big( \cE(\Lambda,k)^c\big)
\,\leq\,
|\Lambda|\times\text{diameter}(\Lambda) 
\times 24
\times
(1-\theta(p))^{k/24}\,.
$$
We take now $\Lambda=\Lambda(n)$ and $k=c\ln n$.
For any $\alpha>0$, we can choose the constant
$c$ sufficiently large so that the righthand side is smaller than 
$cn^{-\alpha}$ for any $n\geq 1$.
\end{proof}
\begin{proposition}
\label{tcube}
Let $n\geq 1$ and suppose that the event $\cE(\Lambda(n),c\ln n)$ occurs.
There exists a constant $c'$ such that
$$\forall x\in\Lambda(n) \setminus\Lambda(n/4)\quad
\exists\,
y \in\Lambda(n/4)\qquad
T_{\Lambda(n)}(x,y)\,\leq\,c'(\ln n)^2\,.$$
\end{proposition}
\begin{proof}
We build iteratively a sequence travelling from $x$ to the box $\Lambda(3n/4)$.
We start from $y_0=x$.
If $x$  belongs to $\din\Lambda(n)$, we choose for $y_1$ a site in
$\Lambda(n)\setminus \din\Lambda(n)$ such that $|x-y_1|\leq 2$. 
If $x$  belongs to 
$\Lambda(n)\setminus \din\Lambda(n)$, we set $y_1=x$.
Suppose that $y_0,\dots, y_m$ have been built in such a way that the following
four conditions are satisfied for any $l\in\{\,1,\dots,m-1\,\}$:
\smallskip

\noindent
$\bullet\quad$ $y_l\in\Lambda(n)\setminus\Lambda(3n/4)$.
\smallskip

\noindent
$\bullet\quad$
$T_{\Lambda(n)}(y_l, y_{l+1})\,\leq\,c\ln n$.
\smallskip

\noindent
$\bullet\quad$
$\forall i\in\{\,1,\dots,6\,\}\qquad
d(y_{l+1},F_i(\Lambda(n)))\,\geq\, d(y_l,F_i(\Lambda(n)))$.
\smallskip

\noindent
$\bullet\quad$
If $h$ is the smallest index such that
$d(y_l,\din\Lambda(n))\,=\, d(y_l,F_h(\Lambda(n)))$, 
then 
$$d(y_{l+1},F_h(\Lambda(n)))\,\geq\, 2d(y_l,F_h(\Lambda(n))) \,.$$
\begin{figure}
\centerline{
\psset{unit=0.45cm}
\begin{pspicture}(-10,-10)(10,10)
\psset{dotscale=1.5}
\psset{dotstyle=o}
\psset{dotstyle=x}
\pspolygon[linewidth=1pt](-10,-10)(-10,10)
(10,10)(10,-10)(-10,-10)
\pspolygon[linewidth=1pt,linestyle=dashed](-7.5,-7.5)(-7.5,7.5)
(7.5,7.5)(7.5,-7.5)(-7.5,-7.5)
\pspolygon[linewidth=2pt](10,3)(10,7)(6,7)(6,3)(10,3)
\uput[-180]{0}(4,4){$F_i^j(\Gamma)$}
\psline[linewidth=1pt]{->}(4,4)(6,4)
\uput[-180]{0}(4,5){$F_i(\Gamma)$}
\psline[linewidth=1pt]{->}(4,5)(6,5)
\uput[-180]{0}(6.5,0){$F_h(\Lambda(n))$}
\psline[linewidth=1pt]{->}(6.5,0)(10,0)
\uput[90]{0}(0,-10){$\Lambda(n)$}
\uput[90]{0}(0,-7.5){$\Lambda(3n/4)$}
\uput[90]{0}(8,3){$\Gamma$}
\uput[90]{0}(0,-0.0){$O$}
\uput[90]{0}(8.2,5){$y_m$}
\psdots(8,5)
\psdots(0,0)
\psset{dotscale=1.5}
\psset{dotstyle=o}
\end{pspicture}
}
\end{figure}
\noindent
If $y_m$ belongs to 
$\Lambda(3n/4)$, the construction terminates. 
Suppose that
$y_m$ does not belong to 
$\Lambda(3n/4)$.
We will next find a site $y_{m+1}$ so that the sequence
$y_0,\dots,y_{m+1}$ satisfies the four conditions above.
Let $\Gamma$ be the largest cubic box centered at $y_m$ included in $\Lambda(n)$.
Let $h$ be the smallest index such that
$$d(y_m,\din\Lambda(n))\,=\, d(y_m,F_h(\Lambda(n))) \,.$$
One face of $\Gamma$ is included in $F_h(\Lambda(n))$. Let $i\in\{\,1,\dots,6\,\}$
be the index such that $F_i(\Gamma)$ is the opposite face. We have then
$$d( F_i(\Gamma), F_h(\Lambda(n)))\,\geq\,
2 d(y_m,F_h(\Lambda(n))) \,.$$
We choose next the index $j$ so that $F_i^j(\Gamma)$ is the 
square included in $F_i(\Gamma)$
which is among the deepest inside the box $\Lambda(n)$. More precisely, we choose
$j$ in $\{\,1,\dots,4\,\}$ such that
$$\forall l\in\{\,1,\dots,6\,\}\qquad
d( F_i^j(\Gamma), F_l(\Lambda(n)))\,\geq\,
 d(y_m,F_l(\Lambda(n))) \,.$$
Since the event
$\cE(\Lambda(n),c\ln n)$ occurs, there exists $y_{m+1}$ in
$F_i^j(\Gamma)$ such that
$$T_{\Lambda(n)}(y_m, y_{m+1})\,\leq\,c\ln n\,.$$
With this choice of $y_{m+1}$,
the sequence
$y_0,\dots,y_{m+1}$ satisfies the four conditions above.
We prove next that the construction stops, i.e., that the sequence
enters the box $\Lambda(3n/4)$ after a finite number of steps.
In fact, every three steps, the distance to the boundary of $\Lambda(n)$
is doubling:
$$\forall l\in\{\,1,\dots,m-3\,\}\qquad
 d(y_{l+3},\din\Lambda(n))\,\geq\,2
 d(y_{l},\din\Lambda(n))
 \,.$$
It follows that
 $$d(y_{m},\din\Lambda(n))\,\geq\,2^{\big\lfloor 
\displaystyle {m-1\over 3}\big\rfloor}
\,.$$
If the construction has not stopped after $m$ steps, then
$y_{m-1}$ is still outside of $\Lambda(3n/4)$, thus
 $$d(y_{m-1},\din\Lambda(n))\,\leq\,\frac{n}{4}+1\,.$$
These two inequalities imply that
 $$
2^{\big\lfloor 
\displaystyle {m-2\over 3}\big\rfloor}
 \,\leq\,\frac{n}{4}+1\,,$$
thus the construction stops at some step $m^*$
satisfying
$m^*\leq c'\ln n$, where $c'$ is a positive constant.
Now the point $y_{m^*}$ is inside the box $\Lambda(3n/4)$ and we have

$$
T_{\Lambda(n)}(x, y_{m^*})\,\leq\,
T_{\Lambda(n)}(x, y_{1})\,+\,
\sum_{0< m<m^*}T_{\Lambda(n)}(y_m, y_{m+1})\,\leq\,3+cc'(\ln n)^2\,.$$
The site $y_{m^*}$ belongs to 
the box $\Lambda(3n/4)$ and its distance to the boundary of $\Lambda(n)$
is larger or equal than $n/4$. By using a few more cubic boxes of side $n/4$
(nine boxes are certainly enough), we can join $y_{m^*}$ to the box
$\Lambda(n/4)$ with a path having at most $9c\ln n$ closed sites.
We concatenate the two paths in order to obtain the desired estimate.
\end{proof}
\section{The tiling of the sphere}
We denote by $S$ the two dimensional sphere of $\R^3$. 
We consider the hyperplanes
of equations:
$$\displaylines{x=0\,,\quad y=0\,,\quad z=0\,,\quad\cr
x=y\,,\quad
x=-y\,,\quad
x=z\,,\quad
x=-z\,,\quad
y=z\,,\quad
y=-z\,.}$$
Let $\cS$ be the set of the orthogonal symmetries with respect to
these hyperplanes.
These hyperplanes induce a tiling of the
sphere $S$ 
into $48$ spherical triangles. We denote by $\cT$ the collection
of these triangles.
The vertices of the triangles of $\cT$ are the vertices of a convex polyhedron
which is a Catalan solid, it is called
the disdyakis dodecahedron or the hexakis octahedron \cite{HO,WK}.
The group of the isometries generated by $\cS$ acts transitively
on the collection $\cT$ of spherical triangles.
Let us consider one of these triangles, for instance the triangle
having for vertices
$$(1,0,0)\,,\quad
\Big(\frac{1}{\sqrt 3},
\frac{1}{\sqrt 3},
\frac{1}{\sqrt 3}\Big)\,,\quad
\Big(\frac{1}{\sqrt 2},
\frac{1}{\sqrt 2},
0\Big)\,.$$
The longest arc of this triangle is the arc joining the
vertices
$(1,0,0)$ and
$\big(1/{\sqrt 3},
{1}/{\sqrt 3},
{1}/{\sqrt 3}\big)$ and its length is
$\arccos
\big(1/{\sqrt 3}\big)<0.96$. 
%
Let $r>0$. 
We define
$$ B_r\,=\,
\big\{\,(x,y,z)\in\Z^3:
x^2+y^2+z^2\leq r^2\,\big\}\,.$$
Let $T$ belong to $\cT$. We define
$$T_r\,=\,
\big\{\,x\in B_r:
d(x, rT)\leq 3\,\big\}\,.$$
We have
$$\forall y,z\in
T_r\qquad
|y-z|\,\leq\,6+r\,\text{diameter}(T)\,\leq\,6+0.96r\,.$$
Therefore
$$\forall r\geq 600\quad
\forall y,z\in
T_r\qquad
|y-z|\,\leq\,0.97 r\,.$$
For any symmetry $s$ in $\cS$, we have
$s\big( B_r)= B_r$ and
$s\big(T_r\big)=s(T)_r$.
The percolation model is invariant under the action of $\cS$, therefore
$$P\big(0\longleftrightarrow
T_r\text{ in }  B_r \big)\,=\,
P\big(0\longleftrightarrow
s(T)_r\text{ in }  B_r \big)\,$$
and the above probability is the same for
any triangle $T$ in $\cT$.
Moreover
$$\din B_r\,\subset\,
\bigcup_{T\in\cT} T_r\,.$$
Proceeding as in the case of the cube, we have, for any $r>0$,
$$T_{ B_r } (0,\din  B_r )\,\geq\,
\min_{T\in\cT} \,
T_{ B_r } (0,T_r )
$$
and, by the FKG inequality,
$$\displaylines{
P\big(
T_{ B_r } (0,\din  B_r )
\geq k\big)
\,\geq\,
P\big( \forall {T\in\cT} \quad
T_{ B_r } (0,T_r )
\geq k\big)\cr
\,\geq\,
\prod_{T\in\cT}
P\big( 
T_{ B_r } (0,T_r )
\geq k\big)
\,.
}$$
By symmetry of the model, all the probabilities appearing in the product
are equal.
It follows then from lemma~\ref{BK} that
$$
\forall {T\in\cT} \quad\forall r>0\qquad
P\big( 
T_{ B_r } (0,T_r )
\geq k\big)
\,\leq\,
(1-\theta(p))^{k/48}\,.$$
We deal next with translates of $ S$.
We consider the event $\cF(\Lambda,k)$ defined as follows:
for any $x\in\Lambda\cap\Z^3$ and any $r>0$ such that $x+B_r\subset\Lambda$ 
and such that
the boundary of $x+B_r$ intersects the lattice $\Z^3$,
the site $x$ can be joined to each of the $48$ sets 
$x+T_r$, $T\in\cT$,
with a path having at most $k$ closed sites.
More precisely,
$$\displaylines{
\cF(\Lambda,k)\,=\,
\Big\{\,\forall x\in\Lambda\cap\Z^3\quad\forall r>0\qquad
\hfill\cr
\hfill
x+B_r\subset\Lambda,\,
(x+\partial B_r)\cap \Z^3\neq\varnothing
\quad
\Longrightarrow
\quad
\forall {T\in\cT} \quad
T_{ B_r } (0,T_r )
\leq k\,
\Big\}
\,.
}$$
\begin{proposition}
\label{dode}
Let $p$ be such that $\theta(p)>0$ and
let $\alpha>0$. There exists a constant 
$c$, depending on $\alpha$ and $p$,
such that
$$\forall n\geq 2\qquad 
P\big( \cF(\Lambda(n),c\ln n)\big)
\,\geq\,1-\frac{c}{n^\alpha}\,.$$
\end{proposition}
\begin{proof}
The important point is to notice that the number of choices for the site $x$
and the radius $r$ is bounded by $|\Lambda(n)|^2$.
The rest of the proof is the same as proposition~\ref{cube}. 
\end{proof}
\begin{proposition}
\label{tdode}
Let $n\geq 1$ and suppose that the event $\cF(\Lambda(n),c\ln n)$ occurs.
There exists a constant $c'$ such that
$$\forall x,y\in\Lambda(n/4)\qquad
T_{\Lambda(n)}(x,y)\,\leq\,c'(\ln n)^2\,.$$
\end{proposition}
\begin{proof}
We build a sequence starting at $x$ and which converges at geometric speed
towards $y$, and which stops when it is at distance less than $600$ from $y$.
We start from $y_0=x$.
Suppose that $y_0,\dots, y_m$ have been built in such a way that 
for any $l\in\{\,1,\dots,m\,\}$:
\smallskip

\noindent
$\bullet\quad$ $y_{l}\in\Lambda(n/4)$.
\smallskip

\noindent
$\bullet\quad$ $|y_{l}-y|\,\leq\,0.97\, |y_{l-1}-y|$.
\smallskip

\noindent
$\bullet\quad$
$T_{\Lambda(n)}(y_{l-1},y_l)\,\leq\,c\ln n$.
\smallskip

\noindent
We build now $y_{m+1}$.
Let $r>0$ be such that
$y$ is on the boundary of $y_m+B_r$. 
Since $y$ and $y_m$ are in $\Lambda(n/4)$, then
$y_m+B_r$ is included in $\Lambda(n)$.
If $r<600$, then 
$|y_{m}-y|\,<\,600$ and the construction is finished.
Suppose that $r\geq 600$.
There exists $T\in\cT$ such that
$y$ is in $y_m+T_r$. 
Since the event $\cF(\Lambda(n),c\ln n)$ occurs, then there exists
$y_{m+1}\in y_m+T_r$ such that
$$T_{y_m+ B_r}(y_m, y_{m+1})\,\leq\,c\ln n\,.$$
Since $r\geq 600$, then 
$$ |y_{m+1}-y|\,\leq\,0.97 |y_{m}-y|\,,$$
and the sequence $y_0,\dots, y_{m+1}$ satisfies the required constraints.
Since the 
sequence converges at geometric speed
towards $y$, after at most $c'\ln n$ steps, where $c'$ is a constant,
it is at distance less than $600$ from $y$
and the construction terminates at some index $m^*\leq c'\ln n$.
Now we have
$$
T_{\Lambda(n)}(x, y)\,\leq\,
\sum_{0\leq m<m^*}T_{\Lambda(n)}(y_m, y_{m+1})\,+\,
T_{\Lambda(n)}(y_{m^*}, y)
\,\leq\,cc'(\ln n)^2+1800\,.$$
By enlarging the constants, we obtain the statement of the proposition.
\end{proof}
\section{Completion of the proof of theorem~\ref{fm}}
\noindent
We need only to prove the statement for $n$ large enough. Indeed,
if it holds for $n\geq N$, we simply enlarge the constant $\kappa$ so that
$\kappa(\ln2)^2\,\geq\,3(2N+1)$. We have then
$$\forall n\leq N\quad
\forall x,y\in\Lambda(n)\quad
T_{\Lambda(n)}(x,y)\,\leq\,\kappa(\ln 2)^2
\,\leq\,\kappa(\ln n)^2
\,.$$
Let $\alpha>0$. By propositions~\ref{cube} and~\ref{dode}, there exists
a constant $c>0$ such that
$$\forall n\geq 2\qquad 
P\big( 
\cE(\Lambda(n),c\ln n)\cap
\cF(\Lambda(n),c\ln n)\big)
\,\geq\,1-\frac{2c}{n^{\alpha+1}}\,.$$
For $n$ large enough, we have 
${2c}/n^{\alpha+1}<1/n^\alpha$.
Suppose now that the events $\cE(\Lambda(n),c\ln n)$ and
$\cF(\Lambda(n),c\ln n)$ occur simultaneously. Let $x,y\in\Lambda(n)$.
By proposition~\ref{tcube}, there exist $x^*,y^*$ in 
$\Lambda(n/4)$ such that
$$T_{\Lambda(n)}(x,x^*)\,\leq\,c'(\ln n)^2\,,\qquad
T_{\Lambda(n)}(y,y^*)\,\leq\,c'(\ln n)^2
\,.$$
By proposition~\ref{tdode}, since $x^*,y^*$ are in
$\Lambda(n/4)$, then
$$T_{\Lambda(n)}(x^*,y^*)\,\leq\,c'(\ln n)^2\,.$$
We conclude that
$$T_{\Lambda(n)}(x,y)\,\leq\,
T_{\Lambda(n)}(x,x^*)\,+\,
T_{\Lambda(n)}(x^*,y^*)\,+\,
T_{\Lambda(n)}(y^*,y)\,\leq\,
3c'(\ln n)^2\,.$$
This holds for any $x,y$ in $\Lambda(n)$, so we are done.
\bibliographystyle{amsplain}
\bibliography{travel}
\end{document}